\documentclass[letterpaper]{article}
\usepackage{amsthm}
\usepackage{amsmath}
\usepackage{amsfonts}
\usepackage{latexsym}
\usepackage{geometry}
\usepackage{enumerate}
\usepackage{csquotes}
\usepackage{url}
\usepackage{amssymb}
\usepackage{appendix}

\emergencystretch=\maxdimen
\newtheorem{Theorem}[equation]{Theorem}
\newtheorem{Lemma}[equation]{Lemma}
\newtheorem{Corollary}[equation]{Corollary}
\newtheorem{Proposition}[equation]{Proposition}

\theoremstyle{definition}
\newtheorem{Definition}[equation]{Definition}
\theoremstyle{remark}

\numberwithin{equation}{section}

\def \W {\mathcal{W}}
\def \N {\mathcal{N}}
\def \dist {{\rm dist}}
\def \Vol {{\rm Vol}}
\DeclareMathOperator{\Rm}{Rm}
\DeclareMathOperator{\Ric}{Ric}

\title{Perelman's entropy on ancient Ricci flows} 
\author{Zilu Ma and Yongjia Zhang}
\numberwithin{equation}{section}

\begin{document}
\maketitle

\begin{abstract}
In \cite{Z2}, the second author proved Perelman's assertion, namely, for an ancient Ricci flow with bounded and nonnegative curvature operator, bounded entropy is equivalent to noncollapsing on all scales. In this paper, we continue this discussion. It turns out that the  curvature operator nonnegativity is not a necessary condition, and we need only to assume a consequence of Hamilton's trace Harnack. Furthermore, we show that this condition holds for steady Ricci solitons with nonnegative Ricci curvature.
\end{abstract}

\section{Introduction}

The entropy formula for the Ricci flow was introduced by Perelman \cite{Per02}, with the help of which he proved the no local collapsing theorem and some other famous theorems, such as the pseudolocality theorem. Along with the entropy formula, Perelman also invented the reduced geometry for the Ricci flow. Both of these two monotonicity formulas have ever since been central techniques in this field. At the first glance, it appears that Perelman's entropy is neatly formulated, whereas the reduced geometry deals with a subsolution---rather than a solution---to the conjugate heat equation. Nevertheless, the analysis of the reduced geometry 
needs only the local geometric information, whereas the analysis of Perelman's entropy requires one to handle the heat equation, whose solution is sensitive to the global geometry. In view of these facts, it is not surprising that the reduced geometry is more tractable in the localization. Indeed, in the construction of the Ricci flow with surgeries, Perelman chiefly applied the reduced geometry.

On the other hand, it turns out that many theorems proved by the reduced geometry method can also be proved by Perelman's entropy method. To begin with, Zhang \cite{ZQ10} showed that in the proof of the Poincar\'e conjecture, the reduced geometry can be replaced with Perelman's entropy. Furthermore, under the condition of \emph{either} Type I curvature bound \emph{or} bounded and nonnegative curvature operator, a noncollapsed Ricci flow always has an asymptotic shrinking soliton (c.f. \cite{Per02} and \cite{N10}). The original proofs of these asymptotic soliton theorems applied the reduced geometry method, yet they can both be proved by implementing Perelman's entropy method (c.f. \cite{CZ11} and \cite{Z2}). Hereby we would like to point out Bamler's recent groundbreaking works \cite{Bam20a}---\cite{Bam20c}, in which he largely refined the analysis of the Nash entropy (the time average of Perelman's entropy) and proved a nice structure theorem for singularity models of the Ricci flow.

In this paper, we continue the discussion initiated in \cite{Z2}, where the second author proved Perelman's assertion (section 11 in \cite{Per02})

\begin{quotation}
We impose one more requirement on the solutions; namely, we fix some $\kappa>0$ and require that $g_{ij}(t)$ to be $\kappa$-noncollapsed on all scales... \emph{It is not hard to show that this requirement is equivalent to a uniform bound on the entropy $S$}, defined as in 5.1 using an arbitrary fundamental solution to the conjugate heat equation.
\end{quotation}
In this assertion Perelman assumes bounded and nonnegative curvature operator. It is still interesting to ask whether the nonnegative curvature operator condition is necessary. In the present paper, we will show that this condition can indeed be relaxed. 

To present the statements of our main results, let us recall the several definitions. Given a complete Ricci flow $(M,g(t))_{t\in I}$, for any $x,y\in M, s,t\in I,s<t,$ we denote by 
    $K(x,t\,|\, y,s)$
the minimal heat kernel coupled with the flow, namely,
\begin{align*}
(\partial_t - \Delta_{g(t)}) K(\cdot,t\,|\,y,s)=0,
& \quad \lim_{t\downarrow s} K(x,t\,|\,y,s)
= \delta_y(x),\\
(-\partial_s - \Delta_{g(s)} + R(\cdot,s)) K(x,t\,|\,\cdot,s)=0,
& \quad \lim_{s\uparrow t} K(x,t\,|\,y,s)
= \delta_x(y).
\end{align*}
See, for example, \cite[Theorem 24.40]{RFTA-III} for the existence of such heat kernels. 

Following \cite{Per02}, we denote by $\Box:=\partial_t -\Delta_{g(t)}$ the heat operator coupled with the Ricci flow $(M,g(t))_{t\in I}$, and by $\Box^*=-\partial_t-\Delta_{g(t)}+R_{g(t)}$ the conjugate heat operator. It follows from the Stokes theorem that for any $u,v\in C^2_c(M\times I)$, i.e., $C^2$ functions over some interval $I\subset [0,T]$ with compact supports, we have
\[
    \frac{d}{dt}\int_M uv\, dg_t
    = \int_M (\Box u) v - u (\Box^* v) \, dg_t,
\]
where we denote by $dg_t$ the volume form induced by the metric $g(t)$. 

\begin{Definition}
Let $(M,g(t))_{t\in[0,T]}$ be a complete Ricci flow. Let 
\begin{equation*}
    u(x,t)=K(x_0,t_0\,|\, x,t)=(4\pi\tau)^{-\frac{n}{2}}e^{-f(x,t)}
\end{equation*} 
be the conjugate heat kernel based at $(x_0,t_0)\in M\times(0,T]$, where $\tau=t_0-t\in(0,t_0]$ is the backward time. Then Perelman's entropy and the Nash entropy based at $(x_0,t_0)$ are respectively defined as
\begin{eqnarray}
\W_{(x_0,t_0)}(\tau)&=&\int_M\Big(\tau\big(|\nabla f|^2+R\big)+f-n\Big)u\,dg_t,
\\
\N_{(x_0,t_0)}(\tau)&=&\int_Mfu\,dg_t-\frac{n}{2},\label{nashentropy}
\end{eqnarray}
for all $\tau\in(0,t_0]$.
\end{Definition}

Perelman's well-known monotonicity formula indicates that, unless $(M,g(t))$ is the static Euclidean space, in which case $\W_{(x_0,t_0)}(\tau)\equiv0$, $\W_{(x_0,t_0)}(\tau)$ is always negative and monotonically decreasing in $\tau$. Furthermore, Perelman's entropy always converges to zero at its base time, namely,
\begin{eqnarray}\label{baseW}
\lim_{\tau\rightarrow 0}\W_{(x_0,t_0)}(\tau)=0.
\end{eqnarray} 
The Nash entropy bears the same monotonicity property and satisfies (\ref{baseW}), since it is known to be the time average of Perelman's entropy
\begin{eqnarray}\label{average}
\N_{(x_0,t_0)}(\tau)=\frac{1}{\tau}\int_0^\tau \W_{(x_0,t_0)}(\eta)\,d\eta.
\end{eqnarray}

In Perelman's assertion, he mentioned the notions of \emph{bounded entropy} and \emph{noncollapsing}. They precisely mean the following.

\begin{Definition}
Let $(M,g(t))_{t\in(-\infty,0]}$ be an ancient solution. Then $g(t)$ is said to have \textit{bounded entropy}, if
\begin{eqnarray}
W:=\inf_{(x,t)\in M\times(-\infty,0];\eta>0}\W_{(x,t)}(\eta)>-\infty,
\end{eqnarray}
where the quantity $W$ is called the \emph{entropy bound}.
\end{Definition}

\begin{Definition}
Let $(M^n,g(t))$ be a complete Ricci flow. Then, $g(t)$ is called \textit{weakly $\kappa$-noncollapsed}, if, for any $r>0$, it holds that $\operatorname{Vol}_{g(t)}\big(B_{g(t)}(x,r)\big)\geq\kappa r^n$ whenever $|{\Rm}|(y,s)\leq r^{-2}$ for any $(y,s)\in B_{g(t)}(x,r)\times[t-r^2,t]$. $g(t)$ is called \textit{strongly  $\kappa$-noncollapsed}, if, for any $r>0$, it holds that $\operatorname{Vol}_{g(t)}\big(B_{g(t)}(x,r)\big)\geq\kappa r^n$ whenever $R_{g(t)}\leq r^{-2}$ on $B_{g(t)}(x,r)$. Here $B_{g(t)}(x,r)$ stands for the $g(t)$-geodesic ball centered at $x$ with radius $r$, and $\operatorname{Vol}_{g(t)}$ stands for the Riemannian volume with respect to $g(t)$.
\end{Definition}

Next, we review several definitions involved with Perelman's reduced geometry. Suppose that $(M,g(t))_{t\in [-T,0]}$ is a complete Ricci flow. Fix a base point $(p_0,t_0)\in M\times (-T,0].$
For any piecewisely smooth  curves $\gamma:[0,\tau]\to M$ with $\gamma(0)=p$ and $ t_0-\tau \ge -T,$ we define
\[
    \mathcal{L}(\gamma)
    := \int_0^\tau \sqrt{s}\left(R_{g(t_0-s)}
    + |\dot \gamma|^2_{g(t_0-s)}\right)(\gamma(s))\, ds.
\]
Then, let
\[
   L(x,\tau):= L_{(p_0,t_0)}(x,\tau) := \inf_{\gamma} \mathcal{L}(\gamma),
\]
where the infimum is taken over all piecewisely smooth curve $\gamma:[0,\tau]\to M$ with $\gamma(0)=p$ and $\gamma(\tau)=x$, and
\[
    \ell(x,\tau) :=
    \ell_{(p_0,t_0)}(x,\tau)
    :=\frac{1}{2\sqrt{\tau}}L_{(p_0,t_0)}(x,\tau) 
\]
is called the \emph{reduced distance} based at $(p_0,t_0).$ Perelman's \emph{reduced volume} based at $(p_0,t_0)$ is defined as
\[
    V_{(p_0,t_0)}(\tau)
    := (4\pi\tau)^{-n/2}
    \int_M \exp\left(-\ell_{(p_0,t_0)}(\cdot, \tau)\right)\, dg_{t_0-\tau},
\]
for all $\tau\in(0,T+t_0]$. $V_{(p_0,t_0)}(\tau)$ is known to be monotonically decreasing in $\tau$.

The main theorem of this paper is the following.

\begin{Theorem}\label{Main_Theorem_1}
Let $(M^n,g(t))_{t\in(-\infty,0]}$ be an ancient solution to the Ricci flow with bounded curvature within each compact time interval. Let $\ell$ be the reduced distance based at $(p,0)$, where $p\in M$ is a fixed point. Assume that there exists $C>0$, such that the following hold.
\begin{enumerate}[(1)]
    \item $|{\Rm}_{g(t)}|\leq CR_{g(t)}$ for all $t\in(-\infty,0]$.
    \item $\displaystyle |\nabla \ell|^2+R\leq\frac{C\ell}{\tau}$ for all $\tau\in(0,\infty)$, where $\tau=-t$ is the backward time.
\end{enumerate}
 Then $g(t)$ is $\kappa$-noncollapsed if and only if it has bounded entropy; the quantity $\kappa$ and the entropy bound are mutually dependent on.  
\end{Theorem}

\noindent\emph{Remarks}: 

\begin{enumerate}
\item Because of the assumptions above, the notions of strong noncollapsing and weak noncollapsing are equivalent.
 \item Assumption (2) is implied by Hamilton's trace Harnack estimate \cite{Ha95}:
\begin{eqnarray}\label{traceharnak}
\frac{\partial R}{\partial t}-2\langle X,\nabla R\rangle+2\Ric(X,X)\geq 0
\end{eqnarray}
for any smooth vector field $X$ on $M$. See the argument in Section 7.2 of \cite{Per02}.
\end{enumerate}

In \cite{Z2}, the second author used the nonnegativity of the curvature operator for two reasons. Let $(M,g(\tau))_{\tau\in[0,\infty)}$ be an ancient Ricci flow with bounded and nonnegative curvature operator, where $\tau$ is the backward time, then
\begin{enumerate}[(1)]
\item Hamilton's trace Harnack implies assumption (2) in the statement of Theorem \ref{Main_Theorem_1}. This inequality implies that the ancient solution is locally Type I wherever the reduced distance is bounded. This fact, along with the noncollapsing assumption, implies the existence of an asymptotic shrinker.
\item If, in addition to the bounded and nonnegative curvature operator condition, the ancient solution is $\kappa$-noncollapsed, then by Perelman's bounded curvature at bounded distance theorem (c.f. Section 11 of \cite{Per02}), the curvature scale $\displaystyle r(x):=R(x)^{-\frac{1}{2}}$ is comparable with the curvature radius $\displaystyle r_{{\Rm}}(x):=\sup\{s:|{\Rm}|\leq s^{-2}\text{ on } B(x,s)\}$, and hence the geometry of the parabolic cube $B_{g(\tau)}(x,r)\times[\tau-r^2,\tau+r^2]$, where $r:=R(x,\tau)^{-\frac{1}{2}}$, is bounded in terms of $r$. This fact is crucial to the point-wise estimate of the conjugate heat kernel.
\end{enumerate}

As we will find out later in this paper, among the above two points, the latter is not as essential as the former. Indeed, the main technique of proving Theorem \ref{Main_Theorem_1} is to replace the curvature scale $R^{-\frac{1}{2}}$ by $r(x)=\sqrt{\tau}\ell^{-\frac{1}{2}}$. The latter scale, though not directly related to the curvature, serves well for the purpose mentioned in point (2) above. Once this is established, we can use a similar argument as in \cite{Z2} to estimate the conjugate heat kernel. In fact, this scale is also implemented in the proof of the quadratic lower bound of $\ell$; see Lemma 3.2 in \cite{Y}. 

The proof of Theorem \ref{Main_Theorem_1} together with some of Bamler's \cite{Bam20a} results on the Nash entropy also implies the following interesting corollary, which is an extension of a result of Xu \cite{Xu17}.

\begin{Corollary}[Entropy uniqueness of the asymptotic shrinker]\label{Corollary1}
Let $(M,g(t))_{t\in(-\infty,0]}$ be a $\kappa$-noncollapsed ancient solution with bounded curvature within each compact time interval. Assume either
\begin{enumerate}[(1)]
    \item $g(t)$ has Type I curvature bound, i.e., there is a constant $C_{\rm I}<\infty$ such that
    \[
        \sup_M |{\Rm}|_{g(t)} \le \frac{C_{\rm I}}{1+|t|},
    \]
    for all $t\le 0,$
    or
    \item $|{\Rm}|_{g(t)}\leq CR_{g(t)}$ for some positive constant $C$ and for all $t\in(-\infty,0]$, and Hamilton's trace Harnack (\ref{traceharnak}) holds on $(M,g(t))$.
\end{enumerate}
Then, all asymptotic shrinkers of $(M,g(t))$ based at any point in $M\times(-\infty,0]$ have the same entropy, which is also equal to the logarithmic of the Gaussian density as defined in \cite{CHI04}. Furthermore, for any $(x_1,t_1)$ and $(x_2,t_2)\in M\times(-\infty,0]$, the following holds.
\begin{eqnarray}\label{ev}
\lim_{\tau\rightarrow\infty}\ \W_{(x_1,t_1)}(\tau)=\lim_{\tau\rightarrow\infty}\ \W_{(x_2,t_2)}(\tau)=\lim_{\tau\rightarrow\infty}\ \N_{(x_1,t_1)}(\tau)=\lim_{\tau\rightarrow\infty}\ \N_{(x_2,t_2)}(\tau)
\\\nonumber
=\lim_{\tau\rightarrow\infty}\ \log V_{(x_1,t_1)}(\tau)=\lim_{\tau\rightarrow\infty}\ \log V_{(x_2,t_2)}(\tau).
\end{eqnarray}
\end{Corollary}

One may naturally ask, why assumption (2) in the statement of Theorem \ref{Main_Theorem_1} is meaningful at all, since Hamilton's trace Harnack is not known to hold without any strong curvature positivity assumption; see \cite{Ha95} and \cite{B}. In the present paper, we also show that this condition holds on steady Ricci solitons assuming only nonnegative Ricci curvature. Recall that a Riemannian manifold $(M^n,g)$ is said to have a Ricci soliton structure, if there is a smooth vector field $X$ such that
\[
    2\Ric + \mathcal{L}_X g = \lambda g,
\]
for some real number $\lambda.$ The Ricci soliton is called shrinking, steady, or expanding if $\lambda$ is positive, zero, or negative, respectively. The Ricci soliton is gradient if $X=\nabla f$ for some smooth function $f$, which is usually called the \emph{potential function}.
Let $(M^n,g,f)$ be a complete gradient Ricci soliton with potential function $f,$ that is,
\[
    \Ric + \nabla^2 f = \tfrac{\lambda}{2}g,
\]
for some real number $\lambda.$
Set $\tau(t)=1-\lambda t,$ and define $\Phi_t$ to be the $1$-parameter family of diffeomorphisms generated by $\frac{1}{\tau(t)} \nabla^{g}f$ with $\Phi_0={\rm id}$. Then
$g(t):=\tau(t)\Phi_t^*g$ is a Ricci flow and is called the \emph{canonical form} of the gradient Ricci soliton.

For the standard properties of Ricci solitons, see, for example, \cite{RFTA-I} and the references therein.

The following theorem is an application of our main theorem.

\begin{Theorem}\label{main_Thm_2}
Let $(M^n,g(\tau))_{\tau\in[0,\infty)}$ be the canonical form of a steady gradient Ricci soliton with nonnegative Ricci curvature, where $\tau$ is the backward time. 
Then Hamilton's trace Harnack inequality \eqref{traceharnak}  holds.
Consequently, if $|{\Rm}|\leq CR<\infty$ for some constant $C$, then $g(\tau)$ is $\kappa$-noncollapsed if and only if it has bounded entropy, where $\kappa$ and the entropy bound are mutually dependent on. In the case where $g(\tau)$ is non-flat and $\kappa$-noncollapsed, it has a
non-flat asymptotic shrinker.
\end{Theorem}

This paper is organized as follows.  In section 2 we show that a noncollapsed ancient solution satisfying the assumptions in Theorem \ref{Main_Theorem_1} has an asymptotic shrinker, and its entropy converges to that of the asymptotic shrinker. In section 3 we verify Bamler's gradient estimates \cite{Bam20a} on noncompact Ricci flows with bounded curvature. In section 4 we show that on an ancient solution, Perelman's entropy and the Nash entropy based at varying points must all converges to the same number; Theorem \ref{Main_Theorem_1} and Corollary \ref{Corollary1} are proved in this section. In section 5 we apply our main theorem to steady solitons with nonnegative Ricci curvature.

\section{Asymptotic shrinking gradient Ricci soliton}

Let $(M,g(\tau))_{\tau\in[0,\infty)}$ be a $\kappa$-noncollapsed ancient Ricci flow with bounded curvature within each compact time interval, where $\tau$ is the backward time, satisfying 
\begin{eqnarray}\label{pinching}
|{\Rm}|\leq CR.
\end{eqnarray}
Here and henceforth in this section, $C$ stands for a positive constant which may differ from line to line. Furthermore, we assume that $p\in M$ is a fixed point such that item (2) in the statement of Theorem \ref{Main_Theorem_1} holds for $\ell$, the reduced distance based at $(p,0)$. That is, we assume the following estimate for $\ell$.
\begin{eqnarray}\label{gradientupper1}
|\nabla \ell|^2+R\leq\frac{C\ell}{\tau}.
\end{eqnarray}
From Perelman's proof of the existence of an asymptotic shrinker (c.f. \cite{Y}), one can easily show that a noncollapsed ancient solution satisfying (\ref{pinching}) and (\ref{gradientupper1}) also has an asymptotic shrinker. Indeed, the conditions above are sufficient to prove quadratic upper and lower bounds for $\ell$ (c.f. Lemma 3.2 in \cite{Y}):
\begin{eqnarray}\label{quadratic}
\ell(x,\tau)\sim \frac{1}{\tau}\dist_\tau^2(x,y)+C,
\end{eqnarray}
where $y\in M$ is a point at which $\ell(\cdot,\tau)$ is bounded by $C$. However, it is worth mentioning that (\ref{quadratic}) itself is not sufficient to show that the integral $\displaystyle (4\pi\tau)^{-\frac{n}{2}}e^{-\ell}$ of the reduced volume is uniformly negligible outside a large ball, yet this fact can be obtained by Hein-Naber's Gaussian concentration theorem \cite{HN14}; see the argument in the line above (\ref{volume}); all these facts are sufficient to show the existence of an asymptotic shrinker.

In this section, we will show that the entropy of the ancient solution converges to that of its asymptotic shrinker. The main result of 
 this section is the following.

\begin{Proposition} \label{entropyconvergence}
Let $(M,g(\tau))_{\tau\in[0,\infty)}$ be a $\kappa$-noncollapsed ancient solution with bounded curvature within each compact time interval and satisfying (\ref{pinching}). Let $p\in M$ be a fixed point and $\ell$ the reduced distance based at $(p,0)$ such that (\ref{gradientupper1}) holds. Then, for any sequence of positive numbers $\tau_i\nearrow\infty$, if the sequence of points $\{x_i\}_{i=1}^\infty\in M$ satisfies
$$
\limsup_{i\rightarrow\infty}\ell(x_i,\tau_i)<\infty,
$$ 
then the following convergence happens after passing to a subsequence
\begin{eqnarray*}
\Big(M,g_i(\tau),(x_i,1),\ell_i\Big)_{\tau\in[1,2]}\rightarrow \Big(M_\infty,g_\infty(\tau),(x_\infty,1),\ell_\infty\Big)_{\tau\in[1,2]},
\end{eqnarray*}
where the limit is the canonical form of a Ricci shrinker, $g_i(\tau)=\tau_i^{-1}g(\tau\tau_i)$, and $\ell_i(\cdot,\tau)=\ell(\cdot,\tau\tau_i)$. Here the Ricci flows converge in the Cheeger-Gromov-Hamilton sense \cite{Ha}, and the functions $\ell_i$ converge in the weak $*W_{\text{loc}}^{1,2}(M_\infty\times[1,2])$ sense as well as in the $C_{\text{loc}}^{0,\alpha}(M_\infty\times[1,2])$ sense, with arbitrarily fixed $\alpha\in(0,1)$.

Furthermore, we have
\begin{eqnarray} \label{theequality}
\lim_{\eta\rightarrow\infty}\W_{(p,0)}(\eta)=\lim_{\eta\rightarrow\infty}\log V_{(p,0)}(\eta)=\log\left(\int_{M_\infty}(4\pi\tau)^{-\frac{n}{2}}e^{-\ell_\infty}dg_{\infty}(\tau)\right),
\end{eqnarray}\label{equality}
where $\W_{(p,0)}(\eta)$ is Perelman's entropy and $V_{(p,0)}(\eta)$ is Perelman's reduced volume, both based at $(p,0)$.
\end{Proposition}

\emph{Remark:} As indicated above, due to (\ref{pinching}), (\ref{gradientupper1}), and (\ref{quadratic}), the existence statement of an asymptotic shrinker in Proposition \ref{entropyconvergence} is almost straightforward.  The main focus of Proposition \ref{entropyconvergence} is the equality (\ref{theequality}). Xu \cite{Xu17} first proved it for Type I noncollapsed ancient solutions, and the second author \cite{Z2} proved it in the bounded and nonnegative curvature operator case.

In the rest of this section, 
\begin{eqnarray}\label{kernel}
u:=(4\pi\tau)^{-\frac{n}{2}}e^{-f}
\end{eqnarray}
denotes the conjugate heat kernel based at $(p,0)$---the same base point as that of $\ell$. Similar to the idea in \cite{Z2}, the major effort of proving (\ref{theequality}) is to obtain the Gaussian upper and lower bounds for $u$.

\begin{Lemma}
\begin{eqnarray}
(4\pi\tau)^{-\frac{n}{2}}e^{-\ell}&\leq&u,\label{subsolution}
\\
\inf_{M}\ell(\cdot,\tau)&\leq&\frac{n}{2}. \label{lmin}
\end{eqnarray}
\end{Lemma}
\begin{proof}
 (\ref{subsolution}) follows from the fact that the left-hand-side is a subsolution to the conjugate heat equation, and when $\tau\rightarrow 0+$, both sides converge to the Dirac delta measure concentrated at $p\in M$. (\ref{lmin}) simply follows from a maximum principle; see \cite{Per02}. 
\end{proof}

Henceforth we will use $p_{\tau}\in M$ to denote a point satisfying $\ell(p_\tau,\tau)\leq \frac{n}{2}$.

\begin{Lemma}
\begin{eqnarray}\label{gradientupper2}
\left|\frac{\partial \ell}{\partial \tau}\right|\leq \frac{C\ell}{\tau}.
\end{eqnarray}
\end{Lemma}

\begin{proof}
This inequality follows from the assumption (\ref{gradientupper1}), in combination with the following well-known formula of Perelman (see Lemma 2.22 in \cite{Y})
\begin{eqnarray*}
2\frac{\partial \ell}{\partial \tau}+|\nabla \ell|^2-R+\frac{\ell}{\tau}=0.
\end{eqnarray*}
\end{proof}

In view of inequalities (\ref{gradientupper1}) and (\ref{gradientupper2}), let us define the scale
\begin{eqnarray*}
r(x,\tau)=\sqrt{\tau}\min\{c_0, \ell(x,\tau)^{-\frac{1}{2}}\},
\end{eqnarray*}
where $c_0\in(0,\frac{1}{2})$ is a small constant which we will determine in the course of the proof. Note that $\ell$ is always positive on ancient solutions, hence $r$ is well-defined. Let us then consider the parabolic neighborhood centered at $(x,\tau)$ with radius $r(x,\tau)$.
\begin{Lemma}\label{withinslice}
If $c_0$ is taken to be small enough, then the following holds. Let $(x_0,\tau_0)\in M\times(0,\infty)$ and $r_0:=r(x_0,\tau_0)$. Then
\begin{eqnarray}
r(y,\tau_0)\geq \frac{4}{5}r_0,
\quad\text{ for all }\quad  y\in B_{g(\tau_0)}(x_0,r_0).
\end{eqnarray}
\end{Lemma}
\begin{proof}
By  (\ref{gradientupper1}) and the definition of $r$, 
if ever $r(\cdot,\tau)<c_0\sqrt{\tau}$,
then the following computation is valid.
\begin{eqnarray*}
|\nabla r|=\tau^{\frac{1}{2}}|\nabla \ell^{-\frac{1}{2}}|=\frac{\tau^{\frac{1}{2}}}{2\ell^{\frac{3}{2}}}|\nabla \ell|\leq\frac{C}{2}\ell^{-1}=\frac{C}{2\tau}r^2,
\end{eqnarray*}
namely,
\begin{eqnarray*}
|\nabla r^{-1}|\leq \frac{C}{2\tau}\quad \text{ whenever } \quad r^{-1}>\frac{1}{c_0\sqrt{\tau}}.
\end{eqnarray*}
Therefore, for any $y\in B_{g(\tau_0)}(x_0,r_0)$, integrating the above inequality along a $g(\tau_0)$-geodesic connecting $x_0$ and $y$, we have
\begin{eqnarray*}
\frac{1}{r(y,\tau_0)}\leq \frac{1}{r_0}+\frac{C}{2\tau_0}r_0=\frac{1}{r_0}+\frac{C r_0^2}{2\tau_0}\frac{1}{r_0}
\leq\left(1+\frac{Cc_0^2}{2}\right)\frac{1}{r_0}\leq\frac{5}{4 r_0},
\end{eqnarray*}
where the last inequality holds if $c_0$ is taken to be small enough. We have also used the fact that $r_0\leq c_0\sqrt{\tau_0}$ by definition.
\end{proof}

\begin{Lemma}\label{variabletime}
If $c_0$ is taken to be small enough, then the following holds. Let $(x_0,\tau_0)\in M\times(0,\infty)$ and $r_0:=r(x_0,\tau_0)$. Then
\begin{eqnarray*}
r(x_0,\tau)\geq\frac{4}{5}r_0 \text{ for all }\tau\in[\tau_0-r_0^2,\tau_0+r_0^2]
\end{eqnarray*}
\end{Lemma}
\begin{proof}
By  (\ref{gradientupper2}) and the definition of $r$, so long as $r(\cdot,\tau)<c_0\sqrt{\tau}$, the following computation is valid
\begin{eqnarray*}
\left|\frac{\partial}{\partial\tau}r\right|&=&\frac{1}{2}\left|\frac{1}{\tau^{\frac{1}{2}}\ell^{\frac{1}{2}}}-\frac{\tau^{\frac{1}{2}}}{\ell^{\frac{3}{2}}}\frac{\partial}{\partial\tau}\ell\right|
\\
&\leq&\frac{C}{\tau^{\frac{1}{2}}\ell^{\frac{1}{2}}}
=\frac{C}{\tau}r,
\end{eqnarray*}
namely,
\begin{eqnarray*}
\left|\frac{\partial}{\partial\tau}\log\left(\frac{r}{\sqrt{\tau}}\right)\right|\leq \frac{C}{\tau}\quad \text{ whenever }\quad \log\left(\frac{r}{\sqrt{\tau}}\right)<\log c_0\ll0.
\end{eqnarray*}
Integrating the above inequality from $\tau_0$ to $\tau\in[\tau_0-r_0^2,\tau_0+r_0^2]$, we have
\begin{eqnarray*}
\log\left(\frac{r(x_0,\tau)}{\sqrt{\tau}}\right)&\geq&\log\left(\frac{r(x_0,\tau_0)}{\sqrt{\tau_0}}\right)-C\left|\int_{\tau}^{\tau_0}\frac{1}{\eta}d\eta\right|
\\
&\geq&\log\left(\frac{r(x_0,\tau_0)}{\sqrt{\tau_0}}\right)-C\log\left(\frac{1+c_0^2}{1-c_0^2}\right),
\end{eqnarray*}
where the last inequality is because of the fact $\tau\in[\tau_0-r_0^2,\tau_0+r_0^2]\subset\big[(1-c_0^2)\tau_0,(1+c_0^2)\tau_0\big]$. Therefore, taking $c_0$ to be small enough, the lemma follows.
\end{proof}

Summarizing the previous two lemmas, we obtain the following result. This shows that the scale $r$ which we defined above serves effectively as a local curvature scale. This fact can make up for the lack of bounded curvature at bounded distance.

\begin{Proposition}\label{basiclocalgeometry}
Let $c_0$ be a small constant such that Lemma \ref{withinslice} and Lemma \ref{variabletime} both hold. Let $(x_0,\tau_0)\in M\times(0,\infty)$ and $r_0=r(x_0,\tau_0)$. Then the following are true.
\begin{enumerate}[(1)]
\item $\displaystyle r\geq\frac{1}{2}r_0 \text{ on } B_{g(\tau_0)}\left(x_0,r_0\right)\times\left[\tau_0-r_0^2,\tau_0+r_0^2\right]$,

\item $\displaystyle |{\Rm}|\leq Cr_0^{-2} \text{ on } B_{g(\tau_0)}\left(x_0,r_0\right)\times\left[\tau_0-r_0^2,\tau_0+r_0^2\right]$,

\item $\displaystyle\left|\partial^p_\tau\nabla^q {\Rm}\right|\leq C_{p,q}r_0^{-2-2p-q} \text{ on } B_{g(\tau_0)}\left(x,\tfrac{1}{2}r_0\right)\times\left[\tau-\tfrac{1}{4}r_0^2,\tau+\tfrac{1}{4}r_0^2\right]$, where $C_{p,q}$ is a constant depending on $p$ and $q$.

\item ${\rm Vol}\Big(B_{g(\tau_0)}(x_0,r_0)\Big)\geq c r_0^n$, where $c$ is a constant depending on $\kappa$.

\end{enumerate}
\end{Proposition}

\begin{proof}
(1) is a combination of Lemma \ref{withinslice} and Lemma \ref{variabletime}. (2) follows from the definition of $r$ and the assumptions (\ref{pinching}) and (\ref{gradientupper1}). (3) follows from the localized Shi's estimates \cite{Sh}. (4) follows from (2) and the noncollapsing assumption.
\end{proof}

The following Gaussian concentration theorem proved by Hein and Naber in \cite{HN14} is the fundamental technique in the proof of the integral upper bound for the conjugate heat kernel. 

\begin{Proposition}[Hein-Naber's Gaussian Concentration]\label{gaussianconcentration}
We have
\begin{eqnarray*}
\nu_\tau(A)\nu_\tau(B)\leq \exp\left(-\frac{1}{8\tau}\dist_\tau^2(A,B)\right),
\end{eqnarray*}
for any subsets $A$ and $B\subset M$ and for all $\tau>0$,  where
\begin{eqnarray*}
\nu_\tau(A)=\int_A u\, dg(\tau)
\end{eqnarray*}
is a probability measure on $M$.
\end{Proposition}

We are ready to prove the following Gaussian upper bound for the conjugate heat kernel.

\begin{Proposition}\label{Gaussianupper}
Let $u$ be the conjugate heat kernel based at $(p,0)$ and $p_\tau\in M$ a time-dependent point satisfying $\ell(p_\tau,\tau)\leq\frac{n}{2}$. Then $u$ has a time-invariant Gaussian upper bound centered at $p_\tau$, namely,
\begin{eqnarray*}
u(x,\tau)\leq\frac{C}{(4\pi\tau)^{\frac{n}{2}}}\exp\left(-\frac{\dist_\tau^2(p_\tau,x)}{C\tau}\right)
\end{eqnarray*}
for all $(x,\tau)\in M\times(0,\infty)$, where $C$ is a constant independent of time.
\end{Proposition}

\begin{proof}
Let $\tau\in (0,\infty)$ be arbitrarily fixed. Then, by (\ref{gradientupper1}) and the noncollapsing assumption, the following hold on $B_{g(\tau)}(p_\tau,\sqrt{\tau})$.
\begin{eqnarray*}
\ell&\leq& C,
\\
|\Rm|\leq CR&\leq& \frac{C}{\tau},
\\
\text{Vol}\Big(B_{g(\tau)}(p_\tau,\sqrt{\tau})\Big)&\geq& c\tau^{\frac{n}{2}}.
\end{eqnarray*}
By (\ref{subsolution}), if we take $A=B_{g(\tau)}(p_\tau,\sqrt{\tau})$, then we have
\begin{eqnarray*}
\nu_\tau(A)\geq\int_{B_{g(\tau)}(p_\tau,\sqrt{\tau})}(4\pi\tau)^{-\frac{n}{2}}e^{-\ell}dg_\tau\geq c.
\end{eqnarray*}

Fix an arbitrary $x\in M$ and let $r_0=r(x,\tau)\leq c_0\sqrt{\tau}$. Then, by applying Proposition \ref{gaussianconcentration} with $A=B_{g(\tau)}(p_\tau,\sqrt{\tau})$ and $\displaystyle B=B_{g(\tau)}\left(x,r_0\right)$, we have
\begin{eqnarray}
\int_{B_{g(\tau)}\left(x,r_0\right)}u\,dg(\tau)&=&\nu_\tau(B)\leq \nu_\tau(A)^{-1}\exp\left(-\frac{1}{8\tau}\dist_\tau^2(A,B)\right)
\\\nonumber
&\leq&C\exp\left(-\frac{1}{8\tau}\big(\dist_\tau(p_\tau,x)-\sqrt{\tau}-r_0\big)^2\right).
\\\nonumber
&\leq&C\exp\left(-\frac{1}{16\tau}\dist^2_{\tau}(p_\tau,x)\right).
\end{eqnarray}

Next, we extend the above integral bound to space-time. (\ref{gradientupper2}) implies that
\begin{eqnarray}\label{l_tau}
\ell(p_\tau,\tau')\leq C \quad\text{ for all }\quad \tau'\in[\tau-r_0^2,\tau]\subset\big[(1-c_0^2)\tau,\tau\big].
\end{eqnarray}
Then, implementing the same argument as before at $\tau'\in[\tau-r_0^2,\tau]\subset\big[(1-c_0^2)\tau,\tau\big]$, with $A=B_{g(\tau')}(p_\tau,\sqrt{\tau'})$ and $B=B_{g(\tau')}(x,r')$, where $r'=r(x,\tau')$, we obtain
\begin{eqnarray}\label{0}
\int_{B_{g(\tau')}(x,r')}u \,dg(\tau')\leq C\exp\left(-\frac{1}{16\tau}\dist^2_{\tau'}(p_\tau,x)\right),
\end{eqnarray}
for all $\tau\in[\tau-r_0^2,\tau]$. Because of Proposition \ref{basiclocalgeometry} (1)(2), we may find a small constant $c_1\in(0,\frac{1}{4})$, such that
\begin{eqnarray*}
B_{g(\tau)}(x,c_1r_0)\subset B_{g(\tau')}\left(x,r'\right) \text{ for all } \tau\in[\tau-r_0^2,\tau].
\end{eqnarray*}
Hence, (\ref{0}) can be rewritten as
\begin{eqnarray}\label{1}
\int_{B_{g(\tau)}(x,c_1r_0)}u\,dg(\tau')\leq C\exp\left(-\frac{1}{16\tau}\dist^2_{\tau'}(p_\tau,x)\right).
\end{eqnarray}
Note that the distance on the right-hand-side is in terms of $g(\tau')$. This can be dealt with by Perelman's distance distortion estimate. Indeed, the definition of $r_0$ and (\ref{l_tau}) imply that $r(p_\tau,\tau')\geq c\sqrt{\tau}\geq cr_0$ for all $\tau'\in[\tau-r_0^2,\tau]$. In combination with the fact that $r(x,\tau')\geq \frac{1}{2} r_0$ for all $\tau'\in[\tau-r_0^2,\tau]$ as well as Proposition \ref{basiclocalgeometry}(2), we have
\begin{eqnarray*}
{\Ric}_{g(\tau')}\leq Cr_0^{-2} \quad\text{ on }\quad B_{g(\tau')}(x,c_1r_0)\cup B_{g(\tau')}(p_\tau,cr_0)
\end{eqnarray*}
for all $\tau'\in [\tau-r_0^2,\tau]$. Hence, by using Lemma 8.3 in \cite{Per02}, we obtain
\begin{eqnarray*}
\frac{d}{ds}\dist_{s}(p_\tau,x)\leq Cr_0^{-2} \quad\text{ for all } s\in\left[\tau-r_0^2,\tau\right].
\end{eqnarray*}
Integrating the above inequality from $\tau'\in[\tau-r_0^2,\tau]$ to $\tau$, we have
\begin{eqnarray}\label{2}
\dist_{\tau'}(p_\tau,x)\geq  \dist_{\tau}(p_\tau,x)-C.
\end{eqnarray}
Combining (\ref{1}) and (\ref{2}), we obtain
\begin{eqnarray}
\int_{\tau-r_0^2}^\tau\int_{B_{g(\tau)}\left(x,c_1r_0\right)}u\,dg_{\tau'}d\tau'\leq Cr_0^2\exp\left(-\frac{1}{16\tau}\dist^2_{\tau}(p_\tau,x)\right).
\end{eqnarray}

Since Proposition \ref{basiclocalgeometry} provides good geometry bounds on the parabolic ball $B_{g(\tau)}\left(x,c_1r_0\right)\times[\tau-r_0^2,\tau]$, we then obtain the pointwise bound by using the standard parabolic mean value inequality (c.f. Lemma 3.1 in \cite{CTY}).
\begin{eqnarray}\label{part1}
u(x,\tau)&\leq&\frac{C}{r_0^{n+2}}\int_{\tau-r_0^2}^\tau\int_{B_{g(\tau)}\left(x,c_1r_0\right)}u\,dg_{\tau'}d\tau'
\\\nonumber
&\leq&Cr_0^{-n}\exp\left(-\frac{1}{16\tau}\dist^2_{\tau}(p_\tau,x)\right).
\\\nonumber
&\leq&C(4\pi\tau)^{-\frac{n}{2}}\exp\left(-\frac{1}{16\tau}\dist^2_{\tau}(p_\tau,x)-C\log\frac{r_0}{\sqrt{\tau}}\right).
\end{eqnarray}
Finally, to deal with the last term in the above formula, recall the definition of $r_0:=r(x,\tau)=\sqrt{\tau}\min\{c_0, \ell(x,\tau)^{-\frac{1}{2}}\}$, we then compute
\begin{eqnarray}\label{part2}
-C\log\big(\frac{r_0}{\sqrt{\tau}}\big)&\leq& C\log\Big(\max\{c_0^{-1},\ell(x,\tau)^{\frac{1}{2}}\}\Big)
\\\nonumber
&\leq&C\log\left(c_0^{-1}+C+C\frac{\dist^2_\tau(p_\tau,x)}{\sqrt{\tau}}\right)
\\\nonumber
&\leq&\frac{1}{32}\frac{\dist^2_\tau(p_\tau,x)}{\tau}+C,
\end{eqnarray}
where we have used the trivial observation that the logarithmic function grows slower than the quadratic function. Precisely, for any $\varepsilon>0$, there exists $C(\varepsilon)>0$, such that
\begin{eqnarray*}
\log x\leq \varepsilon x^2+C(\varepsilon)\text{ for all } x>0;
\end{eqnarray*}
note that the quadratic upper bound of $\ell$ is merely a consequence of (\ref{gradientupper1}). Combining (\ref{part1}) and (\ref{part2}), the conclusion of the proposition follows.

\end{proof}

Proposition \ref{Gaussianupper} has the following immediate implication. 

\begin{Corollary}
For every $\tau\in(0,\infty)$, let $p_\tau\in M$ be a point such that $\ell(p_\tau,\tau)\leq \frac{n}{2}$. Then $u$ has the following Gaussian upper and lower bounds
\begin{eqnarray}\label{uquadratic}
\frac{1}{C(4\pi\tau)^{\frac{n}{2}}}\exp\left(-\frac{C}{\tau}\dist^2_\tau(p_\tau,x)\right)\leq u(x,\tau)\leq\frac{ C}{(4\pi\tau)^{\frac{n}{2}}}\exp\left(-\frac{1}{C\tau}\dist^2_\tau(p_\tau,x)\right),
\end{eqnarray}
for all $(x,\tau)\in M\times(0,\infty)$, where $C$ is a constant independent of $\tau$. Furthermore, the reduced distance $\ell$ satisfies the following quadratic estimates
\begin{eqnarray}\label{lquadratic}
\frac{1}{C\tau}\dist_\tau^2(p_\tau,x)-C\leq \ell(x,\tau)\leq \frac{C}{\tau}\dist_\tau^2(p_\tau,x)+C,
\end{eqnarray}
for all $(x,\tau)\in M\times(0,\infty)$, where $C$ is a constant independent of $\tau$.
\end{Corollary}

\begin{proof}
The second inequality of (\ref{lquadratic}) follows from the assumption (\ref{gradientupper1}); this fact, in combination with (\ref{subsolution}), implies the first inequality of (\ref{uquadratic}). The first inequality of (\ref{lquadratic}) follows from Proposition \ref{Gaussianupper} and (\ref{subsolution}).
\end{proof}

\begin{proof}[Proof of Proposition \ref{entropyconvergence}]
We only prove the equality (\ref{theequality}). As to the existence of the asymptotic shrinker, one may refer to either the comments at the beginning of this section, or the arguments of Proposition 4.2 in \cite{Z2}. First of all, let us fix the sequences of points $\{x_i\}_{i=1}^\infty\subset M$ and positive numbers $\tau_i\nearrow\infty$ as in the statement of the proposition. Then, we define 
\begin{eqnarray*}
f_i(\cdot,\tau)=f(\cdot,\tau\tau_i),\ \ u_i=(4\pi\tau)^{-\frac{n}{2}}e^{-f_i},
\end{eqnarray*}
where, as before,
\begin{eqnarray*}
u:=(4\pi\tau)^{-\frac{n}{2}}e^{-f}
\end{eqnarray*}
is the conjugate heat kernel based at $(p,0)$. We will prove that $\{f_i\}_{i=1}^\infty$ also smoothly converges to the potential function on the asymptotic shrinker.

First, we verify that the base points $x_i$ are equivalent to $p_{\tau_i}$. By (\ref{gradientupper2}), we have
\begin{eqnarray}\label{lup}
\ell(x_i,\tau)\leq \ell(x_i,\tau_i)\left(\frac{\tau}{\tau_i}\right)^C\leq\frac{n}{2}\cdot2^C\leq C \text{ for all } \tau\in \left[\frac{1}{2}\tau_i,2\tau_i\right].
\end{eqnarray}
It then follows from (\ref{lquadratic}) that
\begin{eqnarray*}
\dist_{\tau}(x_i,p_\tau)\leq C\sqrt{\tau_i} \quad\text{ for all }\quad \tau\in \left[\frac{1}{2}\tau_i,2\tau_i\right].
\end{eqnarray*}
Hence, after the parabolic scaling, the Gaussian upper and lower bounds (\ref{uquadratic}) become
\begin{eqnarray}\label{fquadratic}
\frac{1}{C}\cdot\dist^2_{g_i(\tau)}(x_i,x)-C\leq f_i(x,\tau)\leq C\cdot\dist^2_{g_i(\tau)}(x_i,x)+C,
\end{eqnarray}
for all $\tau\in[\frac{1}{2},2]$.

Arguing as Proposition 3.3 in \cite{Z2}, one may immediately obtain uniform growth estimates for the derivatives of $f_i$; indeed, all gradients of $f_i$ have uniformly polynomial growth bounds. Hence we have
\begin{eqnarray*}
f_i&\rightarrow&f_\infty
\\
u_i&\rightarrow& u_\infty=(4\pi\tau)^{-\frac{n}{2}}e^{-f_\infty}
\end{eqnarray*}
locally smoothly on $M\times[1,2]$. For the convergence of the entropy and of the integral of $u_i$, we need to show that the integral of $u_i$ outside a large ball is negligible. Note that unlike the case of \cite{Z2}, we do not have a lower curvature bound to apply volume comparison theorem. Nevertheless, the Gaussian concentration theorem is sufficient. Fixing a large number $\rho\gg 1$, we apply Proposition \ref{gaussianconcentration} to $A=B_{g(\tau)}(x_i,\sqrt{\tau})$ and $B= M\setminus B_{g(\tau)}(x_i,\rho\sqrt{\tau_i})$ for all $\tau\in[\tau_i,2\tau_i]$. Since, by (\ref{lup}), $\ell(x_i,\tau)$ is bounded from above by a constant, we may argue as the beginning of the proof of Proposition \ref{Gaussianupper} to obtain that $\nu_\tau(A)\geq c> 0$. Hence
\begin{eqnarray*}
\nu_\tau(B)\leq\exp\left(-\frac{1}{16\tau}\rho^2\right),
\end{eqnarray*}
namely,
\begin{eqnarray}\label{negligible}
\int_{M\setminus B_{g_i(\tau)}(x_i,\rho)}u_i\,dg_i(\tau)\leq C\exp(-c\rho^2)\quad \text{ for all }\quad \tau\in[1,2].
\end{eqnarray}
With this estimate, one may continue arguing as the proof of \cite[Proposition 4.2]{Z2} to obtain
\begin{eqnarray}
\lim_{\eta\rightarrow\infty}\W_{(p,0)}(\eta)&=&\int_{M_\infty}\Big(\tau\big(|\nabla f_\infty|^2+R_{g_\infty(\tau)}\big)+f_\infty-n\Big)u_\infty \,dg_\infty(\tau).\label{entropylimit}
\\
\lim_{i\rightarrow\infty}\int_Mu_i\,dg_i(\tau)&=&\int_{M_\infty}(4\pi\tau)^{-\frac{n}{2}}e^{-f_\infty} dg_\infty(\tau)=1.\label{masslimit}
\end{eqnarray}
Since the right-hand-side of (\ref{entropylimit}) is a constant, we immediately obtain, by Perelman's monotonicity formula, that $f_\infty$ is the potential function of the asymptotic shrinker
\begin{eqnarray*}
\Ric_\infty+\nabla^2 f_\infty=\frac{1}{2\tau}g_\infty.
\end{eqnarray*}
The quantity on the right-hand-side of (\ref{entropylimit}) is usually called the \emph{entropy of the shrinker}.

To show the second equality of (\ref{theequality}), we need only to verify that the integrand of the reduced volume is negligible outside a large ball. This can be verified by using (\ref{subsolution}) and (\ref{negligible}). Hence we have
\begin{eqnarray}\label{volume}
\int_{M_\infty}(4\pi\tau)^{-\frac{n}{2}}e^{
-\ell_\infty}dg_\infty(\tau)=\lim_{i\rightarrow\infty}\int_M(4\pi\tau)^{-\frac{n}{2}}e^{-\ell_i}dg_i(\tau)=\lim_{\eta\rightarrow\infty} V_{(p,0)}(\eta),
\end{eqnarray}
for all $\tau\in[1,2]$.

The first equality in (\ref{theequality}) is proved by comparing the different normalizations of $\ell_\infty$ and $f_\infty$; note that they are different potential functions of the same shrinker. This was first pointed out by Carrillo and Ni \cite{CN09}. Since we already have (\ref{entropylimit}), (\ref{masslimit}), and (\ref{volume}), one may then argue as the proof of Theorem 1.2 in \cite{Z2}; this finishes the proof of the proposition.

\end{proof}

\section{Bamler's gradient estimate on noncompact manifolds}

In this section, we vindicate Bamler's sharp gradient estimates \cite{Bam20a} on 
noncompact manifolds. Assuming bounded curvature, these results should be readily established by using his method. We include this proof for the convenience of the readers.  

For a function $u$ defined on $M\times [a,b]$, we write
$u_t = u(\cdot, t).$
Following \cite{Bam20a}, we consider the
function 
    $\displaystyle\Phi(x)
    := \frac{1}{\sqrt{4\pi}}
    \int_{-\infty}^{x} \exp(-y^2/4)\, dy.$
Then $\Phi_t(x):= \Phi\left(t^{-1/2}x\right)$
is a solution to the 1-dimensional 
heat equation $\partial_t \Phi_t
=\Phi_t''$
with initial condition 
$\chi_{[0,\infty)}.$
Using Bamler's method, we have
the following.

\begin{Theorem}[Theorem 4.1 in \cite{Bam20a}]
\label{thm:Bamler gradient estimates}
Let $(M,g(t))_{t\in [t_0,t_1]}$ be a complete Ricci flow with
$|{\Rm}|\le \Lambda$ on $M\times [t_0,t_1]$ for some constant $\Lambda<\infty.$
Consider a solution $u$ to the heat equation coupled with $g(t)$ and assume that $u$ takes values in $(0,1).$ 
Let $T\ge 0$ and suppose
$|\nabla \Phi_T^{-1}(u_{t_0})|_{g(t_0)}\le 1$
if $T>0.$

Then 
$|\nabla \Phi_{T+t-t_0}^{-1}(u_t)|_{g(t)} \le 1$ for any $t\in (t_0,t_1].$
\end{Theorem}

Following Bamler, we wish to apply a version of maximum principle for noncompact Ricci flows \cite[Theorem 12.14]{RFTA-II} to $|\nabla \Phi_t^{-1}(u_t)|^2$. To do so an a priori bound is necessary.
The following lemma, whose detailed proof is left to the readers, is a consequence of the Bernstein-Bando-Shi technique (c.f. \cite{Ko07}) as well as the well know Bochner formula for a solution $u$ to the heat equation coupled with the Ricci flow
\begin{eqnarray*}
\Box |\nabla u|^2=-2|\nabla^2 u|^2\leq0,
\end{eqnarray*}
where $\Box = \partial_t - \Delta_{g(t)}$ is the heat operator coupled with the Ricci flow.

\begin{Lemma}\label{Coarse_gradient_estimate}
Let $(M^n,g(t))_{t\in [0,1]}$ be a complete Ricci flow with bounded curvature.
Suppose that $u$ is a solution to the heat 
equation coupled with $g(t)$ 
satisfying $|u|\le 1$ on $M\times [0,1]$.
If $|\nabla u_0|^2 \le A,$ where $A$ is a positive number,
then
$|\nabla u_t|^2 \le A$
for all $t\in [0,1].$
If there is no initial gradient bound,
then we have $t|\nabla u_t|^2 \le 25$
on $M\times [0,1]$.
\end{Lemma}

\def \uep {u^{(\epsilon)}}
\def \hep {h^{(\epsilon)}}
\begin{proof}[Proof of Theorem \ref{thm:Bamler gradient estimates}]
We first consider the case where $T>0.$
By parabolic rescaling and time shifting, we may assume $0<t_0=T<1$ and $t_1\ge 1$. It suffices to show 
$|\nabla \Phi_{1}^{-1}(u_1)|_{g(1)} \le 1$ under the assumption $|\nabla \Phi_{T}^{-1}(u_{T})|_{g(T)} \le 1$. In the following, we shall omit the subindeces and the reader should keep in mind that the norms of the gradients are computed using the evolving metric.

We write $h_t=\Phi_{t}^{-1}(u_t).$
For any small $\epsilon>0,$ let
\[
    \uep_t :=
    \epsilon + (1- 2 \epsilon) u_t
    \in (\epsilon, 1-\epsilon),
    \quad
    \hep_t := \Phi_t^{-1}
    \left(\uep_t\right).
\]
Define $A_\epsilon = \Phi^{-1}(1- \epsilon)$, then
$\left|\hep_t\right| \le A_{\epsilon}\sqrt{t}$ for all $t\in[T,1]$.
Note that
\[
    \uep - \frac{1}{2}
    = (1-2\epsilon)\left(u - \frac{1}{2}\right),
\]
which means that $\uep$ is closer to $1/2$ than $u$.
It follows that 
$\left|\hep_t\right| = \left|\Phi_t^{-1}(\uep_t)\right| \le |\Phi_t^{-1}(u_t)|
=|h_t|$ and
\[
    \left|\nabla \hep_{T}\right|
    =  (1-2\epsilon) 
    \frac{|\nabla u_T|}{\Phi_T'(\hep_{T})}
    \le 
    \frac{|\nabla u_T|}{\Phi_T'(h_{T})}
    = |\nabla h_T|\le 1,
\]
where the first inequality above is due to the definition of $\Phi$, and the last inequality is simply the assumption of the theorem.
Hence
\[
    \left|\nabla \uep_T\right|\le \Phi_T'(\hep_T)
    \le (4\pi T)^{-1/2}.
\]
By Lemma \ref{Coarse_gradient_estimate}, we have 
\[
    \left|\nabla \uep_t\right|\le C
     \quad \text{ for all } \quad t\in [T,1],
\]
where
$C=(4\pi T)^{-1/2}$.
It follows that
\[
    \left|\nabla \hep_t\right|
    \le (4\pi t)^{1/2} C \exp(A_{\epsilon}^2/4),
\]
where we have used the fact that $\left|\hep_t\right|\leq A_{\epsilon}\sqrt{t}$. Hence, $\left|\nabla \frac{\left(\hep_t\right)^2}{2t}\right|$ is bounded on 
$M\times [T,1].$
By the computations in Theorem 4.1 of \cite{Bam20a}, we have
\[
    (\partial_t - \Delta_t)
    \left|\nabla \hep\right|^2
    + \nabla \frac{\left(\hep\right)^2}{2t}
    \cdot \nabla \left|\nabla \hep\right|^2
    = \frac{1- \left|\nabla \hep\right|^2}{t}\left|\nabla \hep\right|^2.
\]
We can apply the maximum principle 
\cite[Theorem 12.14]{RFTA-II} to $|\nabla \hep|^2-1$, and thereby conclude that 
$|\nabla \hep_1|\le 1.$
Taking $\epsilon\to 0$, we have that
$|\nabla h_1|\le 1$ and we have finished the proof in this case.

We now consider the case where $T=0.$ 
This can be proved by a limiting argument as
observed in \cite[Lemma 3.8]{Bam20b}. By parabolic rescaling and time shifting, we may assume
that $T=t_0=0$ and $t_1\ge 1$. It suffices to prove 
$|\nabla h_1|=|\nabla \Phi_1^{-1}(u_1)|\le 1,$
where we use the same notations as above.

By Lemma \ref{Coarse_gradient_estimate}, we have $|\nabla u_s|^2 \le \frac{25}{s}$ on $M\times (0,1]$. 
Let $\epsilon\in (0,1/2)$ be arbitrarily fixed. Then, for any $s\in(0,1/2)$ we will define $T=T(\epsilon,s)\in(0,s)$, so that the result in the first case can be applied. To this end, let us compute as follows.
\[
    \left|\nabla \Phi_T^{-1}(\uep_{s})\right|
    = (4\pi T)^{1/2}
    \exp\left( \left(\Phi_T^{-1}(\uep_s)\right)^2/4T\right) (1-2\epsilon)|\nabla u_s|
    \le C(T/s)^{1/2} \exp(A_{\epsilon}^2/4).
\]
We take $T=T(\epsilon,s)$ small enough so that
$\left|\nabla \Phi_T^{-1}(\uep_s)\right|\le 1.$
By the result in the first case, we have
\[
    \left|\nabla \Phi_{T+1-s}^{-1}(\uep_1)\right| \le 1.
\]
Letting $s\to 0$ and then $\epsilon\to 0,$
we have $|\nabla \Phi_1^{-1}(u_1)| \le 1.$


\end{proof}

We state the following standard result.
\begin{Lemma} \label{lem: distr integrals}
Suppose that $\mu$ is a probabily measure
on some measure space $(X,\mathcal{B}).$
Let $q:X\to \mathbb{R}$ be a measurable function
and let $\phi:\mathbb{R}\to \mathbb{R}$
be absolutely continuous on each compact interval.
For any real number $\alpha,$
if $\phi|_{[\alpha,\infty)}$ is monotone, then
\[  
    \int_{\{q>\alpha\}}
    \phi(q)\, d\mu 
    = \phi(\alpha) \mu(q>\alpha)
    + \int_{\alpha}^\infty
    \phi'(t) \mu(q>t)\, dt.
\]
\end{Lemma}
\begin{proof}
This is a slightly different
statement of \cite[Theorem 8.16]{Ru}.
The proof is the same.
One can consider $\phi'(t)\chi_E(x,t)$ with $E=\{(x,t)\in X\times [\alpha,\infty) :q(x) > t\ge \alpha\}$
and apply Tonelli's theorem to its integral.
\end{proof}

\begin{Proposition}[Proposition 4.2 in \cite{Bam20a}]\label{thegradientintegralestimate}
Suppose that the same condition as in Theorem \ref{thm:Bamler gradient estimates} holds for a complete Ricci flow
$(M,g(\cdot))$ defined on $[s,t].$ 
Write $d\nu = K(x,t\,|\cdot,s)\, dg_s$.
Then for any $x\in M, 1\le p<\infty,$ and any measurable subset $\Omega\subset M,$ we have
\begin{equation}
\label{ineq: Lp gradient est}
    (t-s)^{p/2}
    \int_{\Omega}
    \left(\frac{|\nabla_x K(x,t\,|\,\cdot,s)|}{K(x,t\,|\,\cdot,s)}\right)^p\, d\nu
    \le C(n,p)
    \nu(\Omega)\left(
    - \log (\nu(\Omega)/2)
    \right)^{p/2}.
\end{equation}
Moreover, for any $x\in M$ and $v\in T_xM$ with 
$|v|_{g(t)}=1$, it holds that 
\begin{equation}
\label{ineq: L2 gradient est}
    (t-s)\int_M 
    \left|
        \frac{
        \partial_v K(x,t\,|\, \cdot, s)
    }{K(x,t\,|\, \cdot, s)}
    \right|^2\, d\nu 
    \le 1/2.
\end{equation}
    
\end{Proposition}

\begin{proof}
By parabolic rescaling and time shifting, we may
assume that $[s,t]=[0,1].$
Let $q = \frac{
        \partial_v K(x,1\,|\, \cdot, 0)
    }{K(x,1\,|\, \cdot, 0)}.$ 
The proof of \cite[Proposition 4.2]{Bam20a} 
applies here almost without many variations. 
For completeness, we will sketch the proof and will refer to Bamler's original proof wherever his argument comes through.
Let
\[
    \lambda(\alpha) = \nu(q\ge \alpha),
    \quad
    h(t) = \sup\{ \alpha: \lambda(\alpha)\ge t\},
\]
where we have written
$\nu(q\ge \alpha)
= \nu(\{q\ge \alpha\})$, and the notations such as $\nu(q> \alpha)$, $\nu(q=\alpha)$, etc., shall also be defined accordingly.
Obviously, $\lambda$ and $h$ are non-increasing and left continuous.

    
For any measurable subset $\Omega\subset M$, we have, by Proposition \ref{thm:Bamler gradient estimates},
\begin{eqnarray}\label{00}
    \left|\int_\Omega q\, d\nu\right|
    = \left|
        \partial_v \int_{M}
        \chi_\Omega K(x,1\,|\cdot,0)\, dg_0
    \right|
    \le F(\nu(\Omega)),
\end{eqnarray}
where
    $F(s) = \Phi'(\Phi^{-1}(s)).$
The integration and differentiation are interchangeable by similar (and indeed, much simpler) arguments as in the appendix. 

By the definition of $h$, 
Bamler proved that
\[
    \nu(q\ge \alpha)
    = |h\ge \alpha|,
    \quad
    \nu(q>\alpha)
    = |h>\alpha|,
    \quad
    \nu(q=\alpha) = |h=\alpha|.
\]
It follows directly from Lemma \ref{lem: distr integrals} that
\[
    \int_{ \{h>\alpha\} }
    \phi(h)
    = \int_{\{q>\alpha\}}
    \phi(q)\, d\nu,
\]
whenever $\phi|_{[\alpha,\infty)}$
is monotone and absolutely continuous
on any compact intervals.
\\



\noindent\textbf{Claim:} 
Assume $f:\mathbb{R}\to\mathbb{R}$ 
is absolutely continuous on any compact interval and is monotone on both $[A,\infty)$ and $(-\infty,-A]$ for some $A\in \mathbb{R}_+.$ Note that we do not require $f$ to have the same monotonicity on $[A,\infty)$ and $(-\infty,-A]$. For any $\alpha\in \mathbb{R}\cup
\{\pm \infty\}$
and any measurable subset $\Omega\subset M$ with
\[
    \{q>\alpha\}
    \subset \Omega \subset \{q\ge \alpha\},
\]
it holds that
\begin{equation} \label{eq: Bamler(4.12)}
     \int_{\Omega} f(q)\, d\nu
    = \int_0^{\nu(\Omega)} f(h).
\end{equation}
\begin{proof}[Proof of the claim]

The case where $\alpha=+\infty$ is obvious. We only need to  consider the case where $\alpha\in\mathbb{R}$. This is because the case where $\alpha = -\infty$ follows from the case where $\alpha\in\mathbb{R}$ and the following observation: when $\alpha = -\infty$, we have $\Omega=M$ and
\[
    \int_M f(q)\, d\nu
    = \lim_{\alpha\to -\infty}
    \int_{\{q>\alpha\}} f(q)\, d\nu
    = \lim_{\alpha\to -\infty}
    \int_{\{h>\alpha\}} f(h)
    = \int_0^1 f(h),
\]
where the first and the third equalities can be justified
by applying the monotone convergence theorem to 
$f(q)$ and $f(h)$ using the monotonicity of 
$f|_{(-\infty,-A]}.$  

Fix an arbitrary $\epsilon>0$.
Without loss of generality, we may assume $A>\alpha$, for otherwise we may always enlarge $A$. Then $\{q>A\}\subset  \{q>\alpha\}
\subset \Omega$, and there is a partition 
$\alpha=b_0<b_1<\cdots<b_m=A$ such that
${\rm osc}_{[b_{i-1},b_i]} f \le \epsilon/2$
for all $i=1,\cdots,m.$
As in \cite{Bam20a}, we have
\begin{align}\label{22}
    &\ \left|
    \int_{\Omega} f(q)\, d\nu
    - \int_0^{\nu(\Omega)} f(h)
    \right|\\\nonumber
    \le&\ 
     \left|
    \int_{\Omega\cap \{q=\alpha\}} f(q)\, d\nu
    - \int_{[0,\nu(\Omega)]\cap \{h=\alpha\}} f(h)
    \right|
    + \sum_{i=1}^m 
    \left|
    \int_{\{b_{i-1}<q\le b_i\}} 
    f(q)\, d\nu
    - \int_{
    \{b_{i-1}<h\le b_i\}} f(h)
    \right|\\\nonumber
    & 
    + \left|
        \int_{\{q>A\}} f(q) \, d\nu
        - \int_{\{h>A\}} f(h)
    \right|\\ \nonumber
    \le &\, \epsilon,
\end{align}
where we have applied Lemma \ref{lem: distr integrals}
to conclude that the term on the third line of the above formula is indeed zero.

\end{proof}

By the same arguments as in \cite{Bam20a},
the claim above implies that
for $a\in (0,1),$
\begin{align*}
    ah(a) &\le \int_0^a h
    \le F(a),\\
    (1-a)h(a) &\ge \int_a^1 h
    \ge -F(a).
\end{align*}
When $a\in (0,1/4),$ we also have
\[
    h(a) \le F(a)/a \le 
    C(-\log a)^{1/2},\quad
    h(1-a) \ge -C(-\log a)^{1/2}.
\]
Then, applying the claim to $f(t)=|t|^p$ for any $1\le p<\infty$, \eqref{ineq: Lp gradient est} can be proved using the same arguments as in \cite[Proposition 4.2]{Bam20a}.

To prove (\ref{ineq: L2 gradient est}), let us define
\[
    H(a) := \int_0^a h \le F(a),
\]
and clearly $H(0)=0$.
Applying the claim with $f(t)=t$ and $\alpha=-\infty$, we have
$H(1)=\int_0^1 h = \int_M q\, d\nu = 0.$ From the definition of $H$, we have $H''\le 0$ in the weak sense. Furthermore, we may argue as \cite[Proposition 4.2]{Bam20a} and obtain
\begin{eqnarray}\label{L2}
\int_0^1h(a)^2da\leq \frac{1}{2}.
\end{eqnarray}
Note that the boundary terms
produced by integration by parts 
in \cite[Proposition 4.2]{Bam20a}
vanish because of the fact that
\[
    H'(a) = h(a) \le F(a)/a
    \le C(-\log a)^{1/2},
\]
when $a\in (0,1/4).$
Applying the claim above with $f(t)=t^2$
and $\alpha=-\infty$, (\ref{ineq: L2 gradient est}) follows from (\ref{L2}).

\end{proof}

\section{The Nash entropy based at varying points}

After having established Bamler's gradient estimates on noncompact manifolds, we subsequently prove several results for the Nash entropy. Following \cite{Bam20a}, we use the notation
\begin{eqnarray}
\N_s^*(x,t):=\N_{(x,t)}(t-s)
\end{eqnarray}
to manifest the dependence of the Nash entropy on its base point. Here $\N_{(x,t)}(t-s)$ is as defined in (\ref{nashentropy}). Henceforth we will also define the time-dependent probability measure
\begin{eqnarray*}
\nu_{x,t}(s)(A)=\int_A K(x,t\,|\,\cdot,s)\,dg(s)
\end{eqnarray*}
for all measurable 
subset $A\subset M$. 
The $L^1$ Wassernstein distance is also a very important tool in \cite{Bam20a}:

\begin{Definition}
Let $(X,d)$ be a complete metric space and $\mu$, $\nu$ probability measures on $X$. Then the $L^1$ Wassernstein distance between $\mu$ and $\nu$ is defined as
\begin{eqnarray*}
d_{W_1}(\mu,\nu)=\sup_{f}\left(\int_X fd\mu-\int_X fd\nu\right),
\end{eqnarray*}
where the supremum is taken over all bounded $1$-Lipshcitz functions.
\end{Definition}

 The following observation made by Bamler is merely a consequence of Lemma \ref{Coarse_gradient_estimate}.

\begin{Proposition} [Lemma 2.7 in \cite{Bam20a}]\label{Warmono}
Let $(M,g(t))_{t\in[0,T]}$, where $0<T<\infty$, be a Ricci flow with bounded curvature. Assume that $s$, $t$, $t_1$, and $t_2\in[0, T]$ satisfy $s<t\leq t_1$, $t_2$. Let $x_1$, $x_2\in M$. Then the following holds
\begin{eqnarray*}
d_{W_1}^{g(s)}(\nu_{x_1,t_1}(s),\nu_{x_2,t_2}(s))\leq d_{W_1}^{g(t)}(\nu_{x_1,t_1}(t),\nu_{x_2,t_2}(t)).
\end{eqnarray*}
\end{Proposition}

Our generalization of Bamler's gradient estimate (Proposition \ref{thm:Bamler gradient estimates}) also leads to the generalization of the following theorem and corollary in \cite{Bam20a}. These results indicate that, fixing the evaluating time, the Nash entropy satisfies good properties with respect to its base point.

\begin{Theorem} [Theorem 5.9 in \cite{Bam20a}]\label{thedifferentialestimatesofthenashentropy}
Let $(M,g(t))_{t\in [0,T]}$, where $0<T<\infty$, be a complete Ricci flow with bounded curvature. Let $s\in [0,T)$ and assume $R(\cdot,s)\geq   R_{\operatorname{min}}$, where $R_{\operatorname{min}}$ is a real number. Then, on $M\times (s,T]$, it holds that
\begin{eqnarray*}
|\nabla \N_s^*|\leq\left(\frac{n}{2(t-s)}-R_{\operatorname{min}}\right)^{\frac{1}{2}}, \ \ -\frac{n}{2(t-s)}\leq\left(\frac{\partial}{\partial t}-\Delta_t\right) \N_s^*\leq0.
\end{eqnarray*}
\end{Theorem}
\begin{proof}
The proof of Theorem \ref{interchangeability} shows that the integration and the differentiation are always interchangeable in the computations of both $\nabla \N_s^*(\cdot,t)$ and $\Box_{x,t}\N_s^*(x,t)$. Hence, the proof follows after \cite{Bam20a} line-by-line, using (\ref{ineq: L2 gradient est}) above and Proposition 2.1 in \cite{CMZ21}.
\end{proof}

As a direct consequence, we have the following Corollary.

\begin{Corollary}[Corollary 5.11 in \cite{Bam20a}] \label{Harnack}
In the same settings as the previous theorem, if $R(\cdot, t^*)\geq R_{\operatorname{min}}$ and $s<t^*\leq t_1$, $t_2$, where $s$, $t^*$, $t_1$, and $t_2\in[0,T]$, then for $x_1$ and $x_2\in M$, we have
\begin{eqnarray*}
\N_s^*(x_1,t_1)-\N_s^*(x_2,t_2)\leq\left(\frac{n}{2(t^*-s)}-R_{\operatorname{min}}\right)^{\frac{1}{2}}d_{W_1}^{g(t^*)}(\nu_{x_1,t_1}(t^*),\nu_{x_2,t_2}(t^*))+\frac{n}{2}\log\left(\frac{t_2-s}{t^*-s}\right).
\end{eqnarray*}
\end{Corollary}

We now apply Corollary \ref{Harnack} to an ancient solution. 
\begin{Proposition}\label{asymptoticentropy}
Let $(M,g(t))_{t\in(-\infty,0]}$ be an ancient solution with bounded curvature within each compact time interval. Then for any $(x_1,t_1)$ and $(x_2,t_2)\in M\times(-\infty,0]$, it holds that
\begin{eqnarray}\label{equality}
\lim_{\tau\rightarrow\infty}\W_{(x_1,t_1)}(\tau)=\lim_{\tau\rightarrow\infty}\N_{(x_1,t_1)}(\tau)=\lim_{\tau\rightarrow\infty}\N_{(x_2,t_2)}(\tau)=\lim_{\tau\rightarrow\infty}\W_{(x_2,t_2)}(\tau).
\end{eqnarray}
\end{Proposition}
\begin{proof}
We first prove the second equality of (\ref{equality}). Observe that this is equivalent to
\begin{eqnarray}\label{ne}
\lim_{s\rightarrow-\infty}\big(\N_s^*(x_1,t_1)-\N_s^*(x_2,t_2)\big)=0.
\end{eqnarray}

Without loss of generality, we may assume $t_1\leq t_2$. Let $\varepsilon>0$ be arbitrarily fixed, we then apply Corollary \ref{Harnack} with $s\ll -1$ and $t^*=\varepsilon s\leq t_1$. Since we can take $R_{\text{min}}= 0$ by \cite{CBl}, we have
\begin{eqnarray}\label{med}
\N_s^*(x_1,t_1)-\N_s^*(x_2,t_2)&\leq&\left(\frac{n}{2(1-\varepsilon)|s|}\right)^{\frac{1}{2}}d_{W_1}^{g(t^*)}(\nu_{x_1,t_1}(t^*),\nu_{x_2,t_2}(t^*))
\\\nonumber&&+\frac{n}{2}\log\left(\frac{t_2-s}{t^*-s}\right)
\\\nonumber
&\leq&\left(\frac{n}{2(1-\varepsilon)|s|}\right)^{\frac{1}{2}}d_{W_1}^{g(t_1)}(\delta_{x_1},\nu_{x_2,t_2}(t_1))
\\\nonumber
&&+\frac{n}{2}\log\left(\frac{t_2-s}{(\varepsilon-1)s}\right),
\end{eqnarray}
where we have used Proposition \ref{Warmono}. By the definition of the $L^1$ Warssenstein distance, we have that
\[
d_{W_1}^{g(t_1)}(\delta_{x_1},\nu_{x_2,t_2}(t_1))\leq\int_M \dist_{t_1}(x_1,\cdot)d\nu_{x_2,t_2}(t_1)
\leq C\sum_{j=1}^\infty j\exp(-cj^2)<\infty,
\]
where the second inequality is due to the Gaussian concentration theorem (Proposition \ref{gaussianconcentration}; see also Proposition 2.2 in \cite{CMZ21} under the current assumption). Taking $s\rightarrow-\infty$ on both sides of (\ref{med}), we obtain
\begin{eqnarray*}
\lim_{s\rightarrow-\infty}\big(\N_s^*(x_1,t_1)-\N_s^*(x_2,t_2)\big)\leq \frac{n}{2}\log\left(\frac{1}{1-\varepsilon}\right).
\end{eqnarray*}
Since $\varepsilon>0$ is arbitrary, we then have
\begin{eqnarray*}
\lim_{s\rightarrow-\infty}\big(\N_s^*(x_1,t_1)-\N_s^*(x_2,t_2)\big)\leq0.
\end{eqnarray*}
Reversing the order of $(x_1,t_1)$ and $(x_2,t_2)$, we obtain (\ref{ne}).

(\ref{equality}) is proved in \cite[Corollary 4.5]{Z1}; we include its proof here. Fix an arbitrary $\varepsilon>0$, then (\ref{average}) implies
\begin{eqnarray*}
\W_{(x,t)}(\tau)\leq \N_{(x,t)}(\tau)=\frac{1}{\tau}\int_0^\tau \W_{(x,t)}(\eta)d\eta \leq\frac{1}{\tau}\int_{\varepsilon\tau}^\tau \W_{(x,t)}(\eta)d\eta\leq (1-\varepsilon)\W_{(x,t)}(\varepsilon\tau).
\end{eqnarray*}
By first taking $\tau\rightarrow\infty$ and then $\varepsilon\rightarrow 0$, we obtain
\begin{eqnarray*}
\lim_{\tau\rightarrow\infty}\W_{(x,t)}(\tau)=\lim_{\tau\rightarrow\infty}\N_{(x,t)}(\tau).
\end{eqnarray*}
\end{proof}

Finally, we are ready to prove Theorem \ref{Main_Theorem_1}.

\begin{proof}[Proof of Theorem \ref{Main_Theorem_1}]
Let $(M,g(t))_{t\in(-\infty,0]}$ be an ancient solution satisfying the conditions of this theorem. The sufficiency direction follows immediately from Proposition 3.3 in \cite{Z1}, namely, if $(M,g(t))_{t\in(-\infty,0]}$ has a finite entropy bound, then it is $\kappa$-noncollapsed on all scales, where $\kappa>0$ depends on the entropy bound.

On the other hand, if $(M,g(t))_{t\in(-\infty,0]}$ is $\kappa$-noncollapsed on all scales, where $\kappa>0$, then, by Proposition (\ref{entropyconvergence}), $(M,g(t))_{t\in(-\infty,0]}$ has an asymptotic shrinker. In combination with Proposition \ref{asymptoticentropy}, we have that Perelman's entropy defined on $(M,g(t))_{t\in(-\infty,0]}$ with arbitrary base point must have the same lower bound---the entropy on the asymptotic shrinker. This finishes the proof.
\end{proof}

\begin{proof}[Proof of Corollary \ref{Corollary1}]
We consider only case (2). Let $(x_0,t_0)$ be an arbitrary point in $M\times(-\infty,0]$ and $\ell$ the reduced distance based at $(x_0,t_0)$.
Assumption (2) in Theorem \ref{Main_Theorem_1} is implied by Hamilton's trace Harnack as mentioned in Remarks below Theorem \ref{Main_Theorem_1}.
Then, Corollary \ref{Corollary1} is a combination of Proposition \ref{entropyconvergence} and Proposition \ref{asymptoticentropy}.
\end{proof}

\section{Applications to steady solitons}
In this section, we prove Theorem \ref{main_Thm_2} together with some other results concerning steady solitons with nonnegative Ricci curvature. Let $(M^n,g,f)$ be a complete steady gradient Ricci soliton with nonnegative Ricci curvature, normalized so that
\[
    \Ric = \nabla^2 f,\quad
    R + |\nabla f|^2 = 1.
\]
Throughout this section, we will consider its \emph{canonical form} $g(\tau)$; see section 1 for the definition. Note that $\tau$ stands for the backward time.

We first observe that the Ricci nonnegativity condition implies Hamilton's trace Harnack inequality, which in turn implies item (2) in Theorem \ref{Main_Theorem_1} according to the remarks thereof. After establishing the following lemma, Theorem \ref{main_Thm_2} follows 
immediately from Theorem \ref{Main_Theorem_1}.

\begin{Lemma}
Let  $(M^n,g,f)$ be a complete steady gradient Ricci soliton with nonnegative Ricci curvature and let $(M,g(\tau))_{\tau\in[0,\infty)}$ be its canonical form. 
Then Hamilton's trace Harnack inequality \eqref{traceharnak} holds.
\end{Lemma}
\begin{proof}
Since $g(\tau)=\Phi_\tau^*g$ moves only by diffeomorphisms, it suffices to verify for the metric $g=g(0).$
Recall that 
$
    \nabla R = -2\Ric(\nabla f,\cdot)
$ on steady gradient Ricci solitons; see, for example, (1.27) in \cite{RFTA-I}. (Note that we use a different sign convention on $f$ here.)
Let $X$ be an arbitrary vector field on $M.$
We can then simplify the expression of the trace Harnack quantity as:
\begin{align*}
    &-\left.\frac{\partial R}{\partial \tau}\right|_{\tau=0}-2\langle X,\nabla R\rangle+2\Ric(X,X)\\
    =&\  2 \Ric(\nabla f,\nabla f)
    + 4\Ric(X,\nabla f)
    + 2\Ric(X,X)\\
    =&\  2 \Ric(X+\nabla f, X+\nabla f)
    \ge 0,
\end{align*}
where we have applied the assumption $\Ric\ge 0.$
\end{proof}


Under the assumption of this section, the AVR (asymptotic volume ratio) is well defined,  and hence the existence of a \textit{non-flat} asymptotic shrinker implies zero AVR
as observed in 
\cite{Ni05}, which in turn implies
infinite ASCR (asymptotic scalar curvature ratio) as observed in 
\cite{DZ18}.
Recall that for a Riemannian manifold
$(M,g)$ with nonnegative Ricci curvature, we have
\[
    {\rm AVR}(g)
    := \lim_{r\to \infty}
    \frac{ {\rm Vol}(B_r(p)) }{r^n},
    \quad
    {\rm ASCR}(g)
    := \limsup_{x\to \infty}
    R(x) \dist^2(x,p).
\]
One can easily show that
the two definitions do not depend 
on the choice of the base point 
$p\in M.$ 
\begin{Corollary}
Let $(M,g,f)$ be a complete and
$\kappa$-noncollasped
steady gradient Ricci soliton
with nonnegative Ricci curvature. Assume $|{\Rm}|\le CR$ for 
some constant $C.$
If $g$ is not flat, then
${\rm AVR}(g)=0$ and
${\rm ASCR}(g)=\infty.$
\end{Corollary}

\begin{proof}
Suppose to the contrary that
${\rm AVR}(g)=c>0.$ Let $g(\tau)=\Phi_\tau^*g$ be the canonical form of $(M,g,f)$,
where $\Phi_\tau$ is the 1-parameter family of diffeomorphisms generated by 
$\nabla f.$ Then, by the Bishop-Gromov comparison theorem and the self-similarity of $g(\tau)$, we have
\begin{eqnarray}\label{volumecomparison}
\operatorname{Vol}B_{g(\tau)}(x,r)\geq cr^n
\end{eqnarray}
for all $(x,\tau)\in M\times[0,\infty)$ and for all $r>0$.
Let $\ell$ be the reduced distance
based at some $(p,0).$
Let $\tau_i\nearrow \infty$ and $x_i\in M$
be such that $\ell(x_i,\tau_i)\le n/2$.
By Proposition \ref{entropyconvergence}, after passing to a subsequence, we have that
\[
    (M,g_i(\tau),(x_i,1),\ell_i)
    \to (M_\infty, g_\infty(\tau),
    (x_\infty,1),\ell_\infty),
\]
in the sense as indicated in that proposition, where $g_i(\tau)=\tau_i^{-1}g(\tau_i\tau)$ and $
\ell_i(\cdot,\tau)=\ell(\cdot, \tau_i\tau).$
Here $(M_\infty,g_\infty, \ell_\infty)$ is the canonical form of a non-flat
Ricci shrinker. Obviously, (\ref{volumecomparison}) holds for every $g_i(\tau)$, and consequently it also holds for the asymptotic shrinker. 
It follows that $\operatorname{AVR}(g_\infty)\geq c>0$, which is a 
contradiction to 
the fact that non-flat shrinking gradient Ricci solitons with nonnegative Ricci curvature must have zero AVR (see \cite[Corollary 1.1]{CN09}).
Hence ${\rm AVR}(g)=0.$

To prove 
${\rm ASCR}(g)=\infty$,
we can argue as \cite[Proposition 2.4]{DZ18}.
Suppose that ${\rm ASCR}(g)<\infty.$ 
Write 
$\rho(x):=\dist(p,x)$ 
for some base point $p\in M$.
Then there are constants $A>1$ and $r_0>0$ such that
\[
	R(x) \le 
	\frac{A}{\rho^2(x)}\quad
	\text{ whenever }\quad 
	\rho(x)\ge r_0.
\]
Without loss of generality, 
we assume that the constant $C$ in \eqref{pinching}
is no less than $1$. For large $r>r_0,$
pick $y\in  \partial B(p, 2\sqrt{AC} r )$. Then we have
\[B(y,r)\subset B(p,3\sqrt{AC} r)\setminus B(p,\sqrt{AC}r).
\]
For any $x\in B(y,r),$ we have
\[
	|{\Rm}|(x)\le C R(x) \le \frac{AC}{\rho^2(x)}\le\frac{AC}{ACr^2} = r^{-2}.
\]
Since $g$ is $\kappa$-noncollapsed, we have
\[
	\frac{{\rm Vol}[B(p,3\sqrt{AC} r)]}{(3\sqrt{AC} r)^n}
	\ge \frac{{\rm Vol}[B(y,r)]}{(3\sqrt{AC} r)^n} \ge \frac{\kappa}{(3\sqrt{AC})^n}.
\]
Taking $r\to \infty$, we obtain ${\rm AVR}(g)>0,$ which is a contradiction
to what we just proved.
Hence ${\rm ASCR}(g)=\infty.$

\end{proof}

The above result extends \cite[Theorem 9.44]{CLN06}, while the latter proved 
${\rm ASCR}=\infty$ assuming ${\rm sec}\ge 0,
\Ric>0$, and that $R$ attains its maximum somewhere. Here, apart from $|{\Rm}|\leq CR$, we assume $\kappa$-noncollapsing, a condition which holds for singularity models.
We also generalized the previous 
results in \cite[Theorem 1.10]{CDM20}
to higher dimensions assuming
nonnegative Ricci curvature from a different approach.

Finally, we remark that P.-Y.~Chan \cite{Cha20} recently proved that
for any $4$-dimensional steady gradient Ricci soliton which is a singularity model, we must have that $|{\Rm}|\le C R$ for some constant $C$ (without assuming a curvature decaying condition as in \cite{Cha19}).

\appendix 
\appendixpage

\section{Interchangeablility of integration and differentiation}

In the proofs of Proposition \ref{thegradientintegralestimate} and Theorem \ref{thedifferentialestimatesofthenashentropy}, it must be applied that the integration and the differentiation are interchangeable. In Bamler's \cite{Bam20a} original proof, this is obviously valid since the Ricci flow in question is on a closed manifold. However, under our assumption, we cannot make the same conclusion so easily. Nonetheless, given the curvature boundedness of the Ricci flow, the estimates for the heat kernel have made it possible for us to prove the interchangeability of the integration and the differentiation. We shall only prove the following theorem, and all other similar arguments required in the proofs of Proposition \ref{thegradientintegralestimate} and Theorem \ref{thedifferentialestimatesofthenashentropy} can be proved with similar (and indeed, much easier) method.

\begin{Theorem}\label{interchangeability}
Let $(M^n,g(t))_{t\in I}$ be a complete Ricci flow with bounded curvature within each time interval compact in $I$. Then, for all $s,t\in I$ with $s<t$, we have
\begin{align*}
    \Box_{x,t}\N_s^*(x,t)=-\int_M\Box_{x,t}\big(K(x,t\,|\,\cdot,0)\log K(x,t\,|\,\cdot,s)\big)dg_s-\frac{n}{2t},
\end{align*}
where, as before, we have defined $\N^*_s(x,t):=\N_{(x,t)}(t-s)$.
\end{Theorem}

By a parabolic scaling, we shall assume that $s=0$ and $t=1$. We shall henceforth fix a point $x\in M$, a small convex open neighborhood $x\in U\subset B_{g_0}(x,1)$, and a small positive number $\varepsilon\ll 1$. where $g_0:=g(0)$ will be used as the reference metric throughout the proof. Indeed, we need only to show that
\begin{align*}
    y\rightarrow &\sup_{z\in U}\left|\nabla_z\big(K(z,1\,|\,y,0)\log K(z,1\,|\,y,0)\big)\right|,
    \\
    y\rightarrow &\sup_{z\in U}\left|\nabla^2_z\big(K(z,1\,|\,y,0)\log K(z,1\,|\,y,0)\big)\right|,
    \\
    y\rightarrow& \sup_{t\in(1-\varepsilon,1+\varepsilon)}\left|\partial_t\big(K(x,t\,|\,y,0)\log K(x,t\,|\,y,0)\big)\right|,
\end{align*}
are all dominated by integrable functions, and the conclusion of the theorem follows from Lebesgue's dominated convergence theorem.

\begin{Lemma}\label{lemmaofthecontrolofthefirstderivative}
There are constants $C$, depending on the curvature bound on $M\times[0,1]$ and $\Vol_{g_0}B_{0}(x,1)$, such that
$$\sup_{z\in U}\left|\nabla_z\big(K(z,1\,|\,y,0)\log K(z,t\,|\,y,0)\big)\right|\leq C\exp\left(-C^{-1}\dist^2_0(x,y)\right),$$
for all $y\in M$.
\end{Lemma}
\begin{proof}
Let us first of all fix an arbitrary $y\in M$. By Lemma 26.17 in \cite{RFTA-III}, we have
\begin{eqnarray*}
K(\cdot,\cdot\,|\,y,0)\leq\frac{C}{\Vol_{g_0}B_{0}\left(y,\tfrac{\sqrt{t}}{2}\right)}\leq\frac{C}{t^{\frac{n}{2}}\Vol_{g_0}B_0(y,1)}\leq\frac{C\exp(C\dist_0(x,y))}{t^{\frac{n}{2}}}\quad\text{ on }\quad M\times(0,1],
\end{eqnarray*}
where we have applied the Bishop-Gromov comparison theorem. Then, applying \cite{BCP10, ZQ06}(c.f. Lemma 2.4 in \cite{Z1}) on $M\times [\tfrac{1}{4},1]$, we have
\begin{eqnarray}\label{thefirstcoarsegradientestimate}
\big|\nabla_zK(z,t\,|\,y,0)\big|&\leq&\frac{1}{\sqrt{t-\tfrac{1}{4}}}\cdot K(z,t\,|\,y,0)\cdot\sqrt{\log\left(\frac{C\exp(C\dist_0(x,y))}{K(z,t\,|\,y,0)}\right)}
\\\nonumber
&\leq& CK(z,t\,|\,y,0)\sqrt{\log\left(\frac{C\exp(C\dist_0(x,y))}{K(z,t\,|\,y,0)}\right)},
\end{eqnarray}
for all $(z,t)\in M\times[\tfrac{1}{2},1]$. If we  restrict $(z,t)\in B_0(x,2)\times[\tfrac{1}{2},1]$, then, by Theorem 26.31 in \cite{RFTA-III}, we have
\begin{eqnarray}\label{anothersomewhatusefulgaussianlowerbound}
K(z,t\,|\,y,0)\geq C^{-1}e^{-C\dist^2_0(y,z)}\geq  C^{-1}e^{-C\dist^2_0(x,y)},
\end{eqnarray}
and (\ref{thefirstcoarsegradientestimate}) becomes
\begin{eqnarray}\label{somefinercoarsegradientestimate}
\big|\nabla_zK(z,t\,|\,y,0)\big|\leq CK(z,t\,|\,y,0)\left(C+C\dist^2_0(x,y)\right)
\end{eqnarray}
for all $(z,t)\in B_0(x,2)\times[\tfrac{1}{2},1]$. This further implies that, if $z\in U\subset B_0(x,1)$, then we have
\begin{align}\label{nonsenseintermedia1}
&\, \left|\nabla_z\big(K(z,1\,|\,y,0)\log K(z,t\,|\,y,0)\big)\right|
\\\nonumber
\leq & \ |\nabla_z K(z,1\,|\,y,0)|\big(1+|\log K(z,1\,|\,y,0) |\big)
\\\nonumber
\leq & \ C(C+C\dist^2_0(x,y))\cdot \big(K(z,1\,|\,y,0)+K^{\frac{1}{2}}(z,1\,|\,y,0)+K^2(z,1\,|\,y,0)\big),
\end{align}
where we have applied the fact the fact that $|u\log u|\leq u^{\frac{1}{2}}+u^2$ for all $u>0$. Finally, by Corollary 2.26 in \cite{RFTA-III}, for all $(z,t)\in B_0(x,2)\times[\tfrac{1}{2},1]$, it holds that
\begin{align}\label{nonsenseintermedia2}
    K(z,t\,|\,y,0)\leq Ce^{-C^{-1}\dist^2_0(y,z)}\leq Ce^{-C^{-1}\dist^2_0(x,y)},
\end{align}
where $C$ depends on the curvature bound and $\displaystyle\inf_{z\in B_0(x,2)}\Vol_{g_0}B_0(z,1)$, which in turn depends only on the curvature bound and $\Vol_{g_0}B_0(x,1)$ by the Bishop-Gromov theorem. The lemma then follows from combining (\ref{nonsenseintermedia1}) and (\ref{nonsenseintermedia2}).
\end{proof}

\begin{Lemma}\label{lemmaofthecontrolofthesecondderivative}
There are constants $C$, depending on the curvature bound on $M\times[0,1]$ and $\Vol_{g_0}B_{0}(x,1)$, such that
$$\sup_{z\in U}\left|\nabla^2_z\big(K(z,1\,|\,y,0)\log K(z,t\,|\,y,0)\big)\right|\leq C\exp\left(-C^{-1}\dist^2_0(x,y)\right),$$
for all $y\in M$.
\end{Lemma}
\begin{proof}
Let us fix an arbitrary $y\in M$. Combining (\ref{thefirstcoarsegradientestimate}) and (\ref{nonsenseintermedia2}), we have
\begin{eqnarray}\label{intermedianonsense001}
\big|\nabla_zK(z,t\,|\,y,0)\big|\leq C\exp\left(-C^{-1}\dist^2_0(x,y)\right)\quad\text{ for all }\quad (z,t)\in B_0(x,2)\times[\tfrac{1}{2},1].
\end{eqnarray}
We shall then consider the function $u(z,t):=K(z,t\,|\,y,0)$, which is a solution to the heat equation. Indeed, since
$$\left(\frac{\partial}{\partial t}-\Delta_L\right)\nabla^2u=0,$$
where $\Delta_L$ is the Lichnerowicz Laplacian operator, we have
\begin{eqnarray*}
\Box|\nabla^2u|^2&=&-2|\nabla^3 u|+4\Rm(\nabla^2u,\nabla^2u)
\\
&\leq&-2|\nabla^3 u|^2+C|\nabla^2u|^2,
\end{eqnarray*}
and $C$ depends on the curvature bound. On the other hand, we have
$$\Box |\nabla u|^2=-2|\nabla^2 u|^2.$$
We may then apply Shi's gradient estimate (c.f. Theorem 14.10 in \cite{RFTA-III}, using the cut-off function constructed by Lemma 14.3 therein) on $B_0(x,2)\times[\tfrac{1}{2},1]$ and obtain that
\begin{eqnarray}\label{secondderivative000}
\big|\nabla^2_zK(z,1\,|\,y,0)\big|\leq C\exp\left(-C^{-1}\dist^2_0(x,y)\right)\quad\text{ for all }\quad z\in B_0(x,1).
\end{eqnarray}
Since, for all $z\in B_0(x,1)$, we have 
\begin{align}\label{asecondderivativeestimate}
   &\, \left|\nabla^2_z\big(K(z,1\,|\,y,0)\log K(z,1\,|\,y,0)\big)\right|
   \\\nonumber
   \leq &\ |\nabla^2K(z,1\,|\,y,0)|(|\log K(z,1\,|\,y,0)|+1)+\frac{\left|\nabla_zK(z,1\,|\,y,0)\right|^2}{K(z,1\,|\,y,0)}
   \\\nonumber
   \leq &\ |\nabla^2K(z,1\,|\,y,0)|(|\log K(z,1\,|\,y,0)|+1)+ C\exp\left(-C^{-1}\dist^2_0(x,y)\right),
\end{align}
where in the last inequality, we have applied a consequence of (\ref{somefinercoarsegradientestimate}) and (\ref{nonsenseintermedia2}):
\begin{eqnarray*}
    \frac{\left|\nabla_zK(z,1\,|\,y,0)\right|^2}{K(z,1\,|\,y,0)}&=&\left(\frac{\left|\nabla_zK(z,1\,|\,y,0)\right|}{K(z,1\,|\,y,0)}\right)^2\cdot K(z,1\,|\,y,0)
    \\
    &\leq& C(1+\dist^2_0(x,y))^2\exp\left(-C^{-1}\dist^2_0(x,y)\right)
    \\
    &\leq& C\exp\left(-C^{-1}\dist^2_0(x,y)\right),
\end{eqnarray*}
for all $z\in B_0(x,1)$, the lemma then follows from (\ref{asecondderivativeestimate}) together with (\ref{anothersomewhatusefulgaussianlowerbound}), (\ref{nonsenseintermedia2}), and (\ref{secondderivative000}).
\end{proof}

\begin{Lemma}\label{lemmaofthecontrolofthetimederivative}
There are constants $C$, depending on the curvature bound on $M\times[0,1]$ and $\Vol_{g_0}B_{0}(x,1)$, such that
$$\sup_{t\in(1-\varepsilon,1+\varepsilon)}\left|\partial_t\big(K(x,t\,|\,y,0)\log K(x,t\,|\,y,0)\big)\right|\leq C\exp\left(-C^{-1}\dist^2_0(x,y)\right),$$
for all $y\in M$, where $\varepsilon$ is a small positive constant, which could be taken to be, say $100^{-1}$.
\end{Lemma}
\begin{proof}
Since 
\begin{align*}
    \partial_t\big(K(x,t\,|\,y,0)\log K(x,t\,|\,y,0)\big)=\big(\log K(x,t\,|\,y,0)+1\big)\cdot \Delta_xK(x,t\,|\,y,0),
\end{align*}
the proof is not essentially different from the above lemma.
\end{proof}

\begin{proof}[Proof of Theorem \ref{interchangeability}]
Let $v\in T_x M$ be a unit vector with respect to $g(1)$, and let $\gamma(s)$ be a unit speed $g(1)$-geodesic emanating from $x$ with $\gamma'(0)=v$. Letting $u_z(y):=K(z,1\,|\,y,0)\log K(z,1\,|\,y,0)$, we have
\begin{eqnarray*}
\left|\frac{1}{s}\big(u_{\gamma(s)}(y)-u_x(y)\big)\right|\leq\frac{1}{s}\int_0^s\sup_{z\in B_0(x,1)}|\nabla_zu_z(y)|ds\leq C\exp\left(-C^{-1}\dist^2_0(x,y)\right),
\end{eqnarray*}
for all $s$ small enough, where we have applied (\ref{lemmaofthecontrolofthefirstderivative}), and the right-hand-side is obviously integrable in $y$ because of the curvature boundedness assumption. Hence, by Lebesgue's dominated convergence theorem, we have
\begin{eqnarray*}
\nabla_v\N^*_0(x,1)&=&\lim_{s\rightarrow 0+}\frac{1}{s}\big(\N^*_0(\gamma(s),1)-\N^*_0(x,1)\big)
\\
&=&\lim_{s\rightarrow 0+}\int_M\frac{1}{s}\big(u_{\gamma(s)}(y)-u_x(y)\big)dg_0(y)
\\
&=&\int_M\lim_{s\rightarrow 0+}\frac{1}{s}\big(u_{\gamma(s)}(y)-u_x(y)\big)dg_0(y)
\\
&=&\int_M\nabla_v\big(K(x,1\,|\,y,0)\log K(x,t\,|\,y,0)\big)dg_0(y).
\end{eqnarray*}
We have proved that the first spatial derivative is interchangeable with the integration. In like manner, with the help of Lemma \ref{lemmaofthecontrolofthesecondderivative} and Lemma \ref{lemmaofthecontrolofthetimederivative}, one may also verify that both the second spatial derivative and the first time derivative are interchangeable with the integration. Once this is done, the theorem follows from Bamler's original proof of Theorem 5.9 in \cite{Bam20a}.
\end{proof}

\noindent Department of Mathematics, University of California, San Diego, CA, 92093
\\ E-mail address: \verb"zim022@ucsd.edu"
\\

\noindent School of Mathematics, University of Minnesota, Twin Cities, MN, 55414
\\ E-mail address: \verb"zhan7298@umn.edu"

\end{document}